\documentclass[a4paper]{article}
\usepackage{mathptmx}
\usepackage{amscd,amssymb,amsmath,amsthm}

\makeatletter

\renewcommand{\@seccntformat}[1]
{{\csname the#1\endcsname}.\hspace{0.3em}}

\renewcommand{\section}{\@startsection
{section}
{1}
{0mm}
{-1.5\baselineskip}
{\baselineskip}
{\bfseries\normalsize}}

\renewcommand{\subsection}{\@startsection
{subsection}
{2}
{0mm}
{-\baselineskip}
{0.5\baselineskip}
{\normalsize\itshape}}

\renewcommand{\subsubsection}{\@startsection
{subsubsection}
{3}
{0mm}
{-.5\baselineskip}
{-2mm}
{\normalsize\itshape}}

\makeatother

\theoremstyle{plain}
\newtheorem*{theorem*}{Theorem}
\newtheorem{theorem}{Theorem}[section]
\newtheorem{lemma}{Lemma}[section]
\newtheorem{corollary}[theorem]{Corollary}
\newtheorem{prop}[lemma]{Proposition}
\newtheorem{claim}[lemma]{Claim}
\newtheorem*{corollary*}{Corollary}

\theoremstyle{definition}
\newtheorem*{defin*}{Definition}
\newtheorem{defin}{Definition}[section]

\theoremstyle{remark}

\newtheorem*{quest*}{Question}

\DeclareMathAlphabet{\matheur}{U}{eur}{m}{n}
\DeclareMathAlphabet{\matheus}{U}{eus}{m}{n}
\DeclareMathAlphabet{\matheuf}{U}{euf}{m}{n}

\numberwithin{equation}{section}

\newcommand{\abs}[1]{\left\lvert#1\right\rvert}

\DeclareMathOperator{\Tan}{Tan}

\DeclareMathOperator{\hess}{Hess}

\DeclareMathOperator{\dist}{dist}

\DeclareMathOperator{\grad}{grad}

\begin{document}

\author{Gerasim  Kokarev
\\ {\small\it School of Mathematics, University of Leeds}
\\ {\small\it Leeds, LS2 9JT, United Kingdom}
\\ {\small\it Email: {\tt G.Kokarev@leeds.ac.uk}}
}

\title{On the essential spectra of submanifolds in the hyperbolic space}
\date{\small\it To Nikolai Nadirashvili on his 70th birthday}
\maketitle

\begin{abstract}
\noindent
We study relationships between asymptotic geometry of submanifolds in the hyperbolic space and their regularity properties near the ideal boundary, revisiting some of the related results in the literature. In particular, we discuss hypotheses when minimal submanifolds meet the ideal boundary orthogonally, and compute the essential spectrum of the Laplace operator on submanifolds that are asymptotically close to minimal submanifolds.
\end{abstract}

\medskip
\noindent
{\small

}

%
%
%


\section{Introduction}
\label{intro}

\subsection{Main results and discussion}

Consider a submanifold $\Sigma^m$ in the hyperbolic space $\mathbf H^{n+1}$, where by the latter we mean the complete simply connected $(n+1)$-dimensional Riemannian manifold whose all sectional curvatures are equal to $-1$. The hyperbolic space $\mathbf H^{n+1}$ is naturally identified with the unit ball $\mathbb B^{n+1}$ via the Poincar\'e model, and the boundary at infinity $S^n_\infty(\mathbf H^{n+1})$ corresponds to the boundary sphere $\partial\mathbb B^{n+1}$. We say that a submanifold $\Sigma^m$ is asymptotic to a submanifold $\Gamma^{m-1}\subset S^n_{\infty}(\mathbf H^{n+1})$ if the closure $\overline{\Sigma^{m}}$ in the Euclidean topology intersects the boundary at infinity $S^n_{\infty}(\mathbf H^{n+1})$ precisely by the submanifold $\Gamma^{m-1}$. The following definition introduces a stronger regularity hypothesis on the behaviour $\Sigma^m$ near $\Gamma^{m-1}$.

\begin{defin}
\label{reg:infty}
A submanifold $\Sigma^m\subset\mathbf H^{n+1}$ is called {\em regular at infinity}, if there exists a submanifold $\Gamma^{m-1}\subset S^n_\infty(\mathbf H^{n+1})$ such that the union $\Sigma^{m-1}\cup\Gamma^{m-1}$ is a $C^1$-smooth submanifold with boundary that meets the sphere at infinity $S^n_\infty(\mathbf H^{n+1})$ orthogonally.
\end{defin}

Note that we do not assume that $\Sigma^m$ is a proper submanifold of $\mathbf{H}^{n+1}$, but the regularity hypothesis above forces $\Sigma^m$ to be a proper submanifold in the complement of a sufficiently large ball in $\mathbf{H}^{n+1}$. Principal examples of such submanifolds are provided by the boundary regularity theory of area-minimising currents due to~\cite{HL87,L89b}. More precisely, by~\cite{HL87} for $m=n$ the regular set of any area-minimising rectifiable $m$-current that is asymptotic to a submanifold $\Gamma^{m-1}\subset S^n_\infty(\mathbf H^{n+1})$ is regular at infinity in the sense of Definition~\ref{reg:infty}. In higher codimension the statement is known to hold only for certain area-minimising currents, such as Anderson solutions, see~\cite{L89b}. In addition, as was pointed out in~\cite[Remark~5.2]{Ko23}, any minimal submanifold that extends to a $C^1$-smooth submanifold with boundary in the closure $\bar{\mathbf H}^{n+1}$, obtained by adding the sphere at infinity $\mathbf{H^{n+1}}\cup S^n_\infty(\mathbf H^{n+1})$, has to meet the sphere at infinity orthogonally. This is a common property of minimal submanifolds, generalising the well-known property of geodesics. The argument for it can be extracted from the boundary regularity theory mentioned above, but it seems that there is no explicit discussion of this property in literature. One of the aims of this note is to address this gap. More precisely, in Section~\ref{proof:t1} we prove the following theorem.
\begin{theorem}
\label{t1}
Let $\Sigma^m\subset\mathbf{H}^{n+1}$ be a minimal submanifold that extends to a $C^1$-smooth submanifold with boundary $\Gamma^{m-1}\subset S^n_\infty(\mathbf H^{n+1})$ in the closure $\bar{\mathbf H}^{n+1}$. Then $\Sigma^m$ meets the sphere at infinity $S^n_\infty(\mathbf H^{n+1})$ orthogonally.
\end{theorem}
In Section~\ref{proof:t1} we also give an account on the related background material, aimed at a non-specialist reader. In particular, we discuss properties of the distance function to a submanifold, and obtain relations for the tangent cone of $\Sigma^m$ at a point on the ideal boundary under more general assumptions. In dimension two we are able to compute the boundary tangent cone under mild topological assumptions only.

We proceed with the following definition.
\begin{defin}
\label{am}
A submanifold $\Sigma^m\subset\mathbf H^{n+1}$ is called {\em asymptotically minimal}, if its mean curvature vector $H_x$ satisfies the following relation
$$
\sup\{\abs{H_x}: x\in\Sigma^m\backslash\mathbf{B}_R\}\longrightarrow 0,\qquad\text{as}\quad R\to+\infty,
$$
where $\mathbf{B}_R$ is the hyperbolic ball of radius $R$ centred at a reference point $p\in\mathbf{H}^{n+1}$.
\end{defin}
Clearly, the definition above does not depend on a reference point $p\in\mathbf{H}^{n+1}$. Examples of such submanifolds are ubiquitous. More precisely, as we show in  Section~\ref{exas}, any submanifold $\Sigma^m\subset\mathbf{H}^{n+1}$ that extends to a $C^2$-smooth submanifold with boundary $\Gamma^{m-1}\subset S^n_\infty(\mathbf H^{n+1})$ is asymptotically minimal in the sense of Definition~\ref{am}. Moreover, it turns out that the rate of convergence of the mean curvature  to zero is related to the property of a submanifold to meet the sphere at infinity orthogonally.

Our next aim is to show that the essential spectrum of the Laplace operator on asymptotically minimal manifolds is determined by the regularity hypothesis at infinity. More precisely, the following statement holds.
\begin{theorem}
\label{t2}
Let $\Sigma^m\subset\mathbf H^{n+1}$ be an  asymptotically minimal submanifold that is regular at infinity. Then  the essential spectrum of the Laplace operator on $\Sigma^m$ contains the interval $[(m-1)^2/4,+\infty)$. Moreover, if $\Sigma^m\subset\mathbf H^{n+1}$ is a proper submanifold, possibly with boundary, then the essential spectrum coincides with the interval $[(m-1)^2/4,+\infty)$.
\end{theorem}
When $\Sigma^m$ is not a complete manifold, by the essential spectrum we mean the essential spectrum of the Friedrichs extension of the Laplace operator defined on compactly supported smooth functions on $\Sigma^m$. We refer to~\cite{BD,Ko23} for the related background material. As was shown in~\cite{Ko23}, the spectrum of any minimal submanifold that is regular at infinity has empty discrete part, and coincides with the interval $[(m-1)^2/4,+\infty)$. Thus, Theorem~\ref{t2} can be viewed as a stability statement for the essential spectrum of minimal submanifolds, strengthening general results in~\cite{DL79,Gla66} for this particular setting. For example, any proper submanifold that is $C^2$-smooth up to boundary at infinity $\Gamma^{m-1}$ and touches a $C^1$-smooth up to boundary minimal submanifold at $\Gamma^{m-1}$ satisfies the hypotheses of Theorem~\ref{t2}.

Let us mention that when the submanifold $\Sigma^m$ in Theorem~\ref{t2} is complete and $C^3$-smooth up to boundary $\Gamma^{m-1}$ in $ S^n_\infty(\mathbf H^{n+1})$, it can be viewed as a conformally compact manifold whose sectional curvatures approach $-1$ at $\Gamma^{m-1}$. In this case the conclusion of Theorem~\ref{t2} can be derived from the work~\cite{Mz91}, see also other criteria in~\cite{EW04,Ku97}. However, under our hypotheses  the statement  appears to be new, and the motivation and methods  are very different from the ones in~\cite{Mz91}. In particular, in view of the extremal problem for the bottom of the spectrum, described in~\cite{Ko23}, we see that when $\Sigma^m$ ranges over asymptotically minimal submanifolds that are regular at infinity, the bottom of the spectrum $\lambda_0(\Sigma^m)$ ranges in the interval $[0,(m-1)^2/4]$, and is maximised when the discrete spectrum of $\Sigma^m$ is empty. By Theorem~\ref{t1} and the results in~\cite{Ko23}, the latter occurs, for example,  for minimal submanifolds that are $C^1$-smooth up to boundary at infinity, as well as for regular sets of certain area-minimising rectifiable currents. To our knowledge, the higher regularity up to boundary at infinity is unknown for area-minimising rectifiable currents in codimension greater than one. In certain cases the higher regularity up to boundary might fail even in codimension one, see~\cite{To96}.

\subsection{Organisation of the paper}
Section~\ref{proof:t1} is devoted to the discussion of tangent cones to minimal submanifolds at points on the ideal boundary in the hyperbolic space. We collect the necessary background material, obtain relations for the tangent cone of an arbitrary minimal submanifold, and prove Theorem~\ref{t1}. In this section we also show that the assumptions in Theorem~\ref{t1} on regularity at the ideal boundary can be replaced by topological assumptions in dimension two, see Theorem~\ref{ort:2}. In the next section we obtain an inequality for the Cheeger isoperimetric constant of an annulus in submanifolds with sufficiently small mean curvature, see Theorem~\ref{tc}. It is then used to obtain the lower bound for the bottom of the essential spectrum of asymptotically minimal submanifolds. Here our hypotheses on the ambient manifold are more general -- it is assumed to be a Cartan-Hadamard space whose sectional curvatures are not greater than $-1$. In Section~\ref{exas} we discuss the relationships between the decay of the mean curvature, $C^2$-regularity at the ideal boundary, and the property of meeting the ideal boundary orthogonally. The last section, Section~\ref{proofs}, is devoted to the proof of Theorem~\ref{t2}. The argument is essentially the repetition of the one in~\cite{Ko23}, but it is included in order to outline main ideas and explain differences that have to be made under our hypotheses.

\smallskip
\noindent
{\em Acknowledgements.} I am grateful to the participants of the LMS-Bath Symposium "Advances in Spectral Theory", held in Bath, July 2024, for a number of discussions on related topics. I am also grateful to the referee for careful reading of the paper, and a number of useful comments.

\section{Tangent cones at points on ideal boundary}
\label{proof:t1}

\subsection{Preliminary material}
We start with describing notation and background material necessary for a proof of Theorem~\ref{t1}. Let $\Gamma^k\subset\mathbb R^n$ be a topological submanifold. We say that $\Gamma^k$ is a {\em $k$-dimensional surface} if locally it is the graph of a continuous vector-function. In other words, for any point $p\in\Gamma^{k}$ there exists a continuous vector-function $\varphi:U\to V$, where $U\subset\mathbb R^k$  and $V\subset\mathbb R^{n-k}$ are open subsets, such that 
$$
\{(z,\varphi(z)): z\in U\}=\Gamma^{k}\cap(U\times V),
$$
and this set contains the point $p$. Further, we say that $\Gamma^k$ is {\em differentiable} at a point $p$, if there is a graphical chart containing $p$ such that the defining vector-function $\varphi$ is differentiable at the point $x$, where $p=(x,\varphi(x))$. For such a surface $\Gamma^k$, the tangent space $T_p\Gamma^k$ is defined as the span of the vectors
$$
\bar e_i=(e_i,\frac{\partial\varphi}{\partial x^i}(x)),\qquad\text{where}\quad i=1,\ldots, k,
$$
and $(e_i)$ is the standard basis of $\mathbb R^k$. By a normal vector $\nu$ to $\Gamma^k$ at $p$ we call a vector in $\mathbb R^n$ that is orthogonal to $T_p\Gamma^{k}$. In particular, we may define the following function
$$
r\longmapsto \delta(p,r)=\inf_{\nu}\dist(p+r\nu,\Gamma^k),
$$
where the infimum is taken over all unit normal vectors $\nu$ to $\Gamma^k$ at $p$.

For the sequel we need the following statement, which is briefly mentioned in~\cite{HL87,To96} in the case $k=n-1$. For reader's convenience we also outline the proof below.
\begin{prop}
\label{dist:prop}
Let $\Gamma^k\subset\mathbb R^n$ be a $k$-dimensional surface that is differentiable at a point $p\in\Gamma^{k}$. Then the following limiting relation holds:
$$
\frac{\delta(p,r)}{r}\longrightarrow 1,\qquad\text{when}\quad r\to 0+.
$$
Moreover, if $\Gamma^k$ is a $C^1$-smooth submanifold of $\mathbb R^n$, then for any compact set $K\subset\Gamma^m$ we have
$$
\sup_{p\in K}\abs{\frac{\delta(p,r)}{r}-1}\longrightarrow 0, \qquad\text{when}\quad r\to 0+.
$$
\end{prop}
\begin{proof}[Outline of the proof of Proposition~\ref{dist:prop}]  
First, for a given point $p\in\Gamma^k$ we consider a graphical chart defined by a vector-function $\varphi:U\to\mathbb R^{n-m}$, which is differentiable at a point $x\in U$ such that $p=(x,\varphi(x))$. Then there exists a vector-function $\alpha:U\to\mathbb R^{n-k}$ such that $\alpha(z)\to 0$ as $z\to x$ and the relation
\begin{equation}
\label{eq1}
\abs{p+r\nu-(z,\varphi(z))}^2\geqslant r^2(1-\alpha^2(z))
\end{equation}
holds for all $r>0$ and all unit normal vectors $\nu$ to $\Gamma^k$ at $p$. Indeed, the vector-function $\alpha$ is defined by the relation
$$
\varphi(z)=\varphi(x)+\nabla\varphi(x)\cdot (z-x)+\alpha(z)\abs{z-x},
$$
and inequality~\eqref{eq1} can be checked in a straightforward way by the repeated application of the Cauchy-Schwartz inequality. We proceed with the following claim.
\begin{claim}
\label{c2}
For any $\epsilon>0$ there exists $\delta>0$, which depends on $\varepsilon$, $p$, and $\Gamma^k$, such that the inequality
$$
\frac{1}{r^2}\dist(p+r\nu,\Gamma^k)^2\geqslant 1-\epsilon^2
$$
holds for all $0<r<\delta$ and all unit normal vectors $\nu$ to $\Gamma^k$ at $p$.
\end{claim}
\begin{proof}
First, in the notation above, for a given $\epsilon>0$ we may choose $\delta>0$ such that
$$
\abs{\alpha(z)}<\epsilon\qquad\text{for all}\quad\abs{z-x}<2\delta.
$$
Choosing $\delta>0$ even smaller, if necessary, we may assume that
\begin{equation}
\label{dist:def}
\dist(p+r\nu,\Gamma^k)=\inf_{z\in U}\abs{(x,\varphi(x))+rv-(z,\varphi(z))},
\end{equation}
for $0<r<\delta$. Since the distance $\dist(p+r\nu,\Gamma^k)$ is not greater than the distance between $p+r\nu$ and $p$, we may further assume that the infimum in~\eqref{dist:def} is taken over the subset 
$$
U_1=\{z\in U:\abs{(x,\varphi(x))+r\nu-(z,\varphi(z))}\leqslant r\}.
$$
Note that
$$
\abs{(x,\varphi(x))+r\nu-(z,\varphi(z))}\geqslant\abs{(x,\varphi(x))-(z,\varphi(z))}-r,
$$
and thus, we see that the subset $U_1$ lies in the greater subset $U_2$ that is defined by the relation
$$
U_2=\{z\in U: \abs{(x,\varphi(x))-(z,\varphi(z))}\leqslant 2r\}.
$$
Further, the subset $U_2$ lies in the subset
$$
U_3=\{z\in U: \abs{z-x}\leqslant 2r\},
$$
and thus, for a proof of the claim, it is sufficient to show that
$$
\abs{(x,\varphi(x))+r\nu-(z,\varphi(z))}^2\geqslant r^2(1-\epsilon^2)
$$
for all $z\in U_3$, all $0<r<\delta$, and all unit normal vectors $\nu$ to $\Gamma^k$ at $p$. The latter is a direct consequence of inequality~\eqref{eq1} with our choice of $\delta>0$.
\end{proof}
Note that we always have the inequality
$$
\frac{1}{r^2}\dist(p+r\nu,\Gamma^k)^2\leqslant 1
$$
for all $r>0$ and all unit normal vectors $\nu$. Combining it with Claim~\ref{c2}, we arrive at the first statement in Proposition~\ref{dist:prop}. 

For a proof of the second statement, we may proceed in a similar fashion. First, by a finite covering argument it is sufficient to consider the case when the compact set $K\subset\Gamma^k$ lies in the graph of a vector-function $\varphi$. Second, the function $\alpha$ used above should be viewed now as the function of two variables $(z,x)\in U\times U$. In more detail, if $\varphi$ is $C^1$-smooth, then the function
$$
\alpha(z,x)=\frac{1}{\abs{z-x}}\left(\varphi(z)-\varphi(x)-\nabla\varphi(x)\cdot(z-x)\right)
$$
is continuous on $U\times U$, and vanishes on the diagonal. In particular, for any $\epsilon>0$ we may find $\delta>0$ such that
$$
\abs{\alpha(z,x)}<\epsilon\qquad\text{for all}\quad (z,x)\in\Delta_\delta(K), 
$$
where $\Delta_\delta(K)$ is a finite union of balls in $U\times U$ of radii $\delta$ centred at points $(y,y)$ on the diagonal such that $y\in K$. In particular, we see that for any $(z,x)\in\Delta_\delta(K)$ there exists $y\in K$ such that 
$$
\abs{z-y}+\abs{y-x}<2\delta,
$$
and hence, $\abs{z-x}<2\delta$. Now  following the argument in the proof of Claim~\ref{c2}, we conclude that
$$
1-\epsilon^2\leqslant\frac{1}{r^2}\dist(p+r\nu,\Gamma^k)^2\leqslant 1
$$
for all $0<r<\delta$, all $p\in K$, and all unit normal vectors $\nu$ to $\Gamma^k$ at $p$. Thus, we are done.
\end{proof}

We end this discussion with recalling the notion of tangent cone to a subset $\Sigma$ in $\mathbb R^n$. Following~\cite{F69}, the {\em tangent cone} $\Tan(\Sigma, p)$ is the collection of vectors $w\in\mathbb R^n$ such that there exist sequences $(p_i)$, where $p_i\in\Sigma$, and $(\rho_i)$, where $\rho_i\in\mathbb R$, $\rho_i>0$ such that
$$
p_i\longrightarrow p\qquad\text{and}\qquad\rho_i(p_i-p)\longrightarrow w,
$$
as $i\to+\infty$. The tangent cone is a generalised cone in the sense that it is a closed scale-invariant subset $C\subset\mathbb R^n$, that is $\lambda C=C$ for any $\lambda>0$. As is well-known, and is straightforward to check, if $\Sigma$ is a $C^1$-smooth submanifold in $\mathbb R^n$, then for a point $p\in\Sigma$ the tangent cone $\Tan(\Sigma,p)$ is precisely the tangent space to $\Sigma$ at $p$. Further, if $\Sigma$ is a $C^1$-smooth submanifold with boundary $\Gamma$, and $p\in\Gamma$ is a boundary point, then $\Tan(\Sigma,p)$ coincides with the half-space $\Tan(\Gamma,p)\times \ell^+_p$, where $\ell^+_p$ is the ray spanned by the inward unit normal $\nu$ at $p$, that is $\ell^+_p=\{t\nu: t\geqslant 0 \}$. The hypothesis that $\Sigma$ above is $C^1$-smooth can be weakened. For example, one can assume that $\Sigma$ around a given point $p$ is the graph of a vector-function that is differentiable at $p$, see the discussion before Proposition~\ref{dist:prop}.

\subsection{Proof of Theorem~\ref{t1}}
Following~\cite{HL87}, the main idea is to compute the tangent cone of the minimal submanifold $\Sigma^m\subset\mathbf H^{n+1}$ at a limiting point $p$ from the ideal boundary. In the sequel it is more convenient to use the upper half-space model of the hyperbolic space, identifying it with the set
$$
\{(x,y): x\in\mathbb R^n, y\in\mathbb R, y>0\}.
$$ 
Then the ideal boundary $S^n_\infty(\mathbf H^{n+1})$ is naturally identified with the union $\{y=0\}\cup\{\infty\}$, and without loss of generality, we may assume that $\Sigma^m$ is asymptotic to a closed submanifold $\Gamma^{m-1}$ that lies in the space
\begin{equation}
\label{idb}
\{(x,0):x\in\mathbb R^n\}\simeq\mathbb R^n.
\end{equation}
In the sequel we view the latter as the Euclidean space, equipped with the corresponding distance function. We also denote by $B_r(x,0)$ Euclidean balls centred at points $(x,0)$ in the ideal boundary; the hemispheres $\partial B_r(x,0)\cap\{y>0\}$ are totally geodesic submanifolds in $\mathbf H^{n+1}$.
\begin{claim}
\label{c0}
Let $\Sigma^m\subset\mathbf H^{n+1}$ be a minimal submanifold that is asymptotic to a closed submanifold $\Gamma^{m-1}\subset\{y=0\}$. Then for any point $x\in\mathbb R^n$ and any $0<r<\dist(x,\Gamma^{m-1})$ the intersection $\Sigma^m\cap B_r(x,0)$ is empty.
\end{claim}
\begin{proof}
The statement is a consequence of the convex hull property of minimal submanifolds due to~\cite{An82}. It says that $\Sigma^m$ has to be contained in the convex hull of $\Gamma^{m-1}$, where the latter can be viewed as the intersection of all closed half-spaces in $\mathbf H^{n+1}$. In other words, if the closure of the totally geodesic half-space $B_r(x,0)$ does not contain $\Gamma^{m-1}$, then it does not contain any point of $\Sigma^m$. The statement can be also derived directly from the maximum principle, see~\cite{CM11, JT03}.
\end{proof}

By Proposition~\ref{dist:prop} we conclude that there exists $\rho_\Gamma>0$ such that 
$$
\frac{\delta(p,r)}{r}>\frac{1}{2}\qquad\text{for all}\quad p\in\Gamma\text{ ~and~ } 0<r<2\rho_\Gamma,
$$
and consider the following set
$$
W=\left(\mathbb R^n\times(0,\rho_\Gamma)\right)\backslash\left(\bigcup_{d(x)>2\rho_\Gamma} B_{2\rho_\Gamma}(x,0)\bigcup_{d(x)\leqslant 2\rho_\Gamma} B_{d(x)}(x,0)\right),
$$
where the distance $d(x)=\dist(x,\Gamma^{m-1})$ is the Euclidean distance on the space given by~\eqref{idb}. Note that by Claim~\ref{c0}  the piece of $\Sigma^m$ formed by points near the ideal boundary lies in the set $W$, that is
\begin{equation}
\label{in0}
\Sigma^m\cap\left(\mathbb R^n\times(0,\rho_\Gamma)\right)\subset W,
\end{equation}
and hence, the tangent cone $\Tan(\Sigma^m,(x,0))$ lies in $\Tan(W,(x,0))$ for any $x\in\Gamma^{m-1}$. Further, by the definition of $W$ the inclusion
\begin{equation}
\label{in1}
\Gamma^{m-1}\times (0,\rho_\Gamma)\subset W
\end{equation}
holds. The proof of the next claim follows closely the line of argument in~\cite{HL87}.
\begin{claim}
\label{c01}
Let $\{(x_i,y_i)\}$ be a sequence in $W$ such that $y_i\to 0+$ as $i\to+\infty$. Then $d(x_i)/y_i\to 0$, where $d(x)=\dist(x,\Gamma^{m-1})$ is the Euclidean distance in $\mathbb R^n$.
\end{claim}
\begin{proof}
First, note that by the definition of the set $W$ we always have $d(x_i)\to 0$ as long as $y_i\to 0+$ as $i\to+\infty$. If there is a subsequence $\{(x_{i_\ell},y_{i_\ell})\}$ such that $y_{i_\ell}\geqslant d(x_{i_\ell})^{3/4}$, then
$$
d(x_{i_\ell})/y_{i_\ell}\leqslant d(x_{i_\ell})^{1/4}\to 0.
$$
Thus, the statement is proved for such subsequences, and we may assume that $y_i<d(x_i)^{3/4}$ for all $i$ from now on.

For each $i$ denote by $p_i$ a point in $\Gamma^{m-1}$ such that $d(x_i)=\abs{x_i-p_i}$. Then there exist unit normal vectors $\nu_i$ such that $x_i=p_i+d(x_i)\nu_i$. We define
$$
r_i=d(x_i)+\frac{y_i^2}{d(x_i)}\qquad\text{for all}\quad i.
$$
First, since $y_i<d(x_i)^{3/4}$, we immediately see that
\begin{equation}
\label{ri}
r_i<d(x_i)+d(x_i)^{1/2}.
\end{equation}
Second, it is straightforward to check that the triangle $\Delta_i$ in $\mathbb R^{n+1}$ with the vertices $(p_i,0)$, $(x_i,y_i)$, and $(p_i+r_i\nu_i,0)$ has the right angle at $(x_i,y_i)$. Finally, denoting by $\delta_i$ the Euclidean distance between $(x_i,y_i)$ and $(p_i+r_i\nu_i,0)$ from similarity of triangles or by straightforward calculation, we obtain
$$
\frac{d(x_i)}{y_i}=\frac{\sqrt{r_i^2-\delta_i^2}}{\delta_i}.
$$
Thus, for a proof of the claim it is sufficient to show that the ratio $r_i/\delta_i$ converges to $1$. To prove the latter we note that the inequalities
$$
\delta(p_i,r_i)\leqslant d(p_i+r_i\nu_i)\leqslant\delta_i\leqslant r_i
$$
hold, where we used the definition of $W$ in the second inequality, and the fact that $\Delta_i$ is right-angled in the last. Re-arranging gives
$$
\frac{\delta(p_i,r_i)}{r_i}\leqslant\frac{\delta_i}{r_i}\leqslant 1,
$$
and since  $r_i\to0+$ as $i\to+\infty$ by inequality~\eqref{ri}, Proposition~\ref{dist:prop} now yields the claim.
\end{proof}

The last statement allows us to compute the tangent cone to the set $W$. More precisely, the following claim holds.
\begin{claim}
\label{c02}
For any point $x\in\Gamma^{m-1}$ the tangent cone $\Tan(W,(x,0))$ coincides with the product
$$
\Tan(W,(x,0))=\Tan(\Gamma^{m-1},x)\times [0,+\infty).
$$
\end{claim}
\begin{proof}
The inclusion
$$
\Tan(\Gamma^{m-1},x)\times [0,+\infty)\subset \Tan(W,(x,0))
$$
is a direct consequence of relation~\eqref{in1}. To prove the converse inclusion we start with sequences $\{(x_i,y_i)\}$ and $\{\rho_i\}$, where $(x_i,y_i)\in W$ and $\rho_i>0$, such that
$$
(x_i,y_i)\longrightarrow (x,0)\qquad\text{and}\qquad \rho_i((x_i,y_i)-(x,0))\longrightarrow (w,y)\in\Tan(W,(x,0)).
$$
In particular, the coordinate $y$ has to be non-negative, and it remains to show that the vector $w$ lies in the tangent cone $\Tan(\Gamma^{m-1},x)$. Consider the sequence $\{p_i\}$, where $p_i\in\Gamma^{m-1}$ is chosen such that $d(x_i)=\abs{x_i-p_i}$. We claim that
$$
\abs{p_i-x}\longrightarrow 0\qquad\text{and}\qquad\rho_i(p_i-x)\longrightarrow w,
$$
and thus, indeed $w$ lies in the tangent cone to $\Gamma^{m-1}$. The first limiting relation follows by triangle inequality together with the observation that $d(x_i)=\abs{x_i-p_i}\to 0$, see the definition of the set $W$. The second is a consequence of Claim~\ref{c01},
$$
\rho_i(p_i-x)=\rho_i(x_i-x)+\rho_i(p_i-x_i)=\rho_i(x_i-x)+\rho_iy_i\frac{d(x_i)}{y_i}
\longrightarrow w+0.
$$
Thus, we are done.
\end{proof}

Now combining relation~\eqref{in0} with Claim~\ref{c02}, we immediately obtain
\begin{equation}
\label{tan:in1}
\Tan(\Sigma^m,(x,0))\subset T_x\Gamma^{m-1}\times[0,+\infty).
\end{equation}
In particular, if $\Sigma^m\cup\Gamma^{m-1}$ is a $C^1$-smooth submanifold with boundary, then $\Sigma^m$ meets the space $\{y=0\}\subset S^n_\infty(\mathbf H^{n+1})$ orthogonally. This ends the proof of Theorem~\ref{t1}.
\qed

\subsection{Other results and extensions}
Let $\Sigma^m\subset\mathbf H^{n+1}$ be an arbitrary minimal submanifold asymptotic to a smooth submanifold $\Gamma^{m-1}\subset S^n_\infty(\mathbf H^{n+1})$. Note that the tangent cones $\Tan(\Sigma^m,(x,0))$ and $\Tan(\overline{\Sigma^m},(x,0))$ coincide, and since $\Gamma^{m-1}$ lies in the closure $\overline\Sigma^m$, together with~\eqref{tan:in1} we immediately arrive at the following relations
\begin{equation}
\label{tan:in2}
T_x\Gamma^{m-1}\times\{0\}\subset\Tan(\Sigma^m,(x,0))\subset T_x\Gamma^{m-1}\times[0,+\infty),
\end{equation}
where we view $\mathbf H^{n+1}$ as an upper half-space and assume $x\in\Gamma^{m-1}$. Removing  suitable regions from a subspace $\mathbf{H}^m\subset\mathbf{H}^{n+1}$, it is straightforward to construct a variety of examples of open minimal submanifolds $\Sigma^m$ such that the first inclusion in~\eqref{tan:in2} is an equality, as well as examples when both inclusions are strict.

The hypothesis that $\Sigma^m\cup\Gamma^{m-1}$ is $C^1$-smooth up to boundary, used in Theorem~\ref{t1} to ensure that the equality
\begin{equation}
\label{tcone}
\Tan(\Sigma^m,(x,0))=T_x\Gamma^{m-1}\times[0,+\infty)
\end{equation}
holds, can be weakened. For example, for a given point $x\in\Gamma^{m-1}$ it is sufficient to assume that around it the set $\Sigma^m\cup\Gamma^{m-1}$ is the graph of a vector-function that is differentiable at $x$, see the discussion after Proposition~\ref{dist:prop}. As the following result shows, it appears that in dimension two it is sufficient to ask only topological assumptions on $\Sigma^m\cup\Gamma^{m-1}$ for relation~\eqref{tcone} to hold.
\begin{theorem}
\label{ort:2}
Let $\Sigma^2\subset\mathbf H^{n+1}$ be a minimal surface asymptotic to an one-dimensional submanifold $\Gamma^{1}\subset S^n_\infty(\mathbf H^{n+1})$, and suppose that the union $\Sigma^2\cup\Gamma^{1}$ equipped with Euclidean topology is a topological surface with boundary. Then for any $x\in\Gamma^{1}$ the tangent cone $\Tan(\Sigma^2,(x,0))$ satisfies relation~\eqref{tcone}.
\end{theorem}
\begin{proof}
By the second inclusion in~\eqref{tan:in2} it is sufficient to show that the product $T_x\Gamma^{1}\times[0,+\infty)$ lies in the tangent cone $\Tan(\Sigma^2,(x,0))$. Besides, by the first inclusion in~\eqref{tan:in2} we know so does the cone $T_x\Gamma^{1}\times\{0\}$. Now we show that the cone $\{0\}\times [0,+\infty)$ also lies in $\Tan(\Sigma^2,(x,0))$.

For a given $x\in\Gamma^{1}$ let $N_x$ be the subspace normal to $\Gamma^{1}$ at $x$ in $\mathbb R^n$, and denote by $V$ the vector subspace in $\mathbb R^{n+1}$ spanned by $N_x$ and the vector $(0,1)$. First, we claim that the intersection of $\Sigma^2$ with the affine subspace $(x,0)+V$ in $\mathbb R^{n+1}$ is non-empty.
Indeed, note that $(x,0)+V$ is a hyperspace in $\mathbb R^{n+1}$, and thus, separates $\mathbb R^{n+1}$ in two half-spaces. If $\Sigma^2$ does not intersect the affine space $(x,0)+V$, then so does not a neighbourhood $\Sigma^2\cap B_r((x,0))$. Choosing a sufficiently small $r$, by the hypotheses of the theorem, we may assume that $\Sigma^2\cap B_r((x,0))$ is connected, and hence, its closure lies in one of the closed half-spaces. However, the arc $\Gamma^1\cap B_r(x)$ lies in the closure of $\Sigma^2\cap B_r(x,0)$, and since it is orthogonal to $(x,0)+V$, it is separated by it. Thus, there are non-trivial arcs of $\Gamma^1$ that lie in each of the half-spaces, and we arrive at a contradiction. 

The argument above also shows that the intersection $\Sigma^2\cap ((x,0)+V)$ contains $(x,0)$ among its limit points. In particular, there exists a sequence $(x_i,y_i)$ from this intersection that converges to $(x,0)$. Since it lies in the set $W$, the argument in the proof of Claim~\ref{c01} shows that $\abs{x-x_i}/y_i$ converges to zero. Now we conclude that 
$$
\frac{1}{y_i}\left((x_i,y_i)-(x,0)\right)\longrightarrow (0,1),
$$
as $i\to+\infty$, and hence, the whole ray $\{0\}\times [0,+\infty)$ lies in the tangent cone $\Tan(\Sigma^2,(x,0))$.

Arguing in a similar fashion, one can also show that for any vector $w\in T_x\Gamma\times [0,+\infty)$ the cone $\{tw: t\geqslant 0\}$ lies in the tangent cone $\Tan(\Sigma^2,(x,0))$. In more detail, the intersection of $\Sigma^2$ with the subspace $(x,0)+V$, where now we denote by $V$ the subspace spanned by $w$ and $N_x$, contains $(x,0)$ among its limit points, that is there is a sequence $(x_i,y_i)$ that converges to $(x,0)$. Writing it in in the form
$$
(x_i,y_i)=(x,0)+\alpha_i w+\beta_i\nu_i,
$$ 
where $\nu_i\in N_x$ is a unit normal vector, we see that the argument in the proof of Claim~\ref{c01} yields the convergence of $\abs{\beta_i}/y_i$ to zero. Note that, if $w_0>0$ is the Euclidean dot product of $w$ and the vector $(0,1)$, then $y_i=\alpha_iw_0$. By the last relation, we obtain
$$
\frac{w_0}{y_i}\left((x_i,y_i)-(x,0)\right)=w+w_0\frac{\beta_i}{y_i}\nu_i\longrightarrow w,
$$
as $i\to+\infty$. Thus, the theorem is proved.
\end{proof}

\section{An estimate for the bottom of essential spectrum}
\label{mckean}

\subsection{An inequality for the Cheeger constant}
Now let $M$ be a complete non-compact Riemannian manifold, and $\Sigma^m$ a submanifold in $M$. For a geodesic ball $B_r=B_r(p)$ in $M$ we define the quantity
$$
\epsilon_r(p)=\sup\{\abs{H_x}: x\in\Sigma^m\backslash B_r(p)\},
$$
where $H_x$ is the mean curvature vector of $\Sigma^m$. If the complement $\Sigma^m\backslash B_r(p)$ is empty, we set $\epsilon_r(p)$ to be $+\infty$. Following Definition~\ref{am}, we call  a submanifold $\Sigma^m$ {\em asymptotically minimal} if $\epsilon_r(p)\to 0$ as $r\to +\infty$ for some, and hence any, point $p\in\mathbf{H}^{n+1}$. Further, for any domain $D\subset\Sigma^m$ by its {\em isoperimetric constant} $h(D)$ we call the quantity
$$
h(D)=\inf\{\mathit{Area}(\partial\Omega)/\mathit{Vol}(\Omega): \bar\Omega\subset D\},
$$
where $\Omega$ ranges over all open submanifolds of $D$ with compact closure and whose boundary is smooth. Note that when $D$ is a domain in the hyperbolic space $\mathbf H^m$, using the classical isoperimetric inequality, it is straightforward to show that the isoperimetric constant satisfies the bound
\begin{equation}
\label{iso:hm}
h(D)>m-1,
\end{equation} 
see the discussion in~\cite{Ko23}. We proceed with the following estimate for the isoperimetric constant on annuli.
\begin{theorem}
\label{tc}
Let $M$ be a complete simply connected Riemannian manifold whose sectional curvatures are not greater than $-1$, and let $\Sigma^m$ be a submanifold in $M$. Then for any concentric balls $B_r(p)\subset B_R(p)$ in $M$  the inequality
$$
h(\Sigma^m\cap(B_R(p)\backslash B_r(p)))\geqslant h(\mathbf B^m_R)-\epsilon_r(p)
$$
holds, where $\mathbf B^m_R$ is a ball in the hyperbolic space $\mathbf H^m$. In particular,  the isoperimetric constant of the complement of a ball $B_r(p)$ in $\Sigma^m$ satisfies the inequality
$$
h(\Sigma^m\backslash B_r(p))\geqslant m-1-\epsilon_r(p)
$$
\end{theorem}

As a consequence of Theorem~\ref{tc}, we obtain the following lower bound for the bottom of the essential spectrum of the Laplace operator on asymptotically minimal submanifolds.
\begin{corollary}
\label{ess}
Let $M$ be a complete simply connected Riemannian manifold whose sectional curvatures are not greater than $-1$, and let $\Sigma^m$ be a proper asymptotically minimal submanifold of $M$, possibly with boundary. Then the bottom of the essential spectrum $\lambda_{\mathit{ess}}(\Sigma^m)$ of the Laplace operator on $\Sigma^m$ satisfies the inequality
$$
\lambda_{\mathit{ess}}(\Sigma^m)\geqslant\frac{1}{4}(m-1)^2.
$$
\end{corollary}
\begin{proof}
For an open subset $D\subset\Sigma^m$ by $\lambda_0(D)$ below we denote the bottom of the Dirichlet spectrum of the Laplace operator. First, by Cheeger's inequality, see~\cite{Cha}, for sufficiently large both $R>r$ we obtain
$$
\lambda_0(\Sigma^m\cap(B_R(p)\backslash B_r(p)))\geqslant\frac{1}{4}h(\Sigma^m\cap(B_R(p)\backslash B_r(p)))^2\geqslant \frac{1}{4}(m-1-\epsilon_r(p))^2,
$$
where in the second inequality we used Theorem~\ref{tc}, and the choice of $r$ so that $\epsilon_r(p)$ is not greater than $m-1$. By standard monotonicity of $\lambda_0$, we may pass to the limit as $R\to+\infty$, and arrive at the inequality
$$
\lambda_0(\Sigma^m\backslash B_r(p))\geqslant \frac{1}{4}(m-1-\epsilon_r(p))^2.
$$
Now by the so-called decomposition principle, see~\cite{DL79,Gla66}, the essential spectra of $\Sigma^m$ and $\Sigma^m\backslash B_r(p)$, where the latter is equipped with the Dirichlet boundary conditions, coincide, and hence, we have the following bound
$$
\lambda_{\mathit{ess}}(\Sigma^m)\geqslant\lambda_0(\Sigma^m\backslash B_r(p))\geqslant\frac{1}{4}(m-1-\epsilon_r(p))^2.
$$
Since $\epsilon_r(p)\to 0+$ as $r\to+\infty$, passing to the limit in the last inequality,  we immediately arrive at the inequality in the statement.
\end{proof}

\subsection{Proof of Theorem~\ref{tc}}
We follow closely the line of argument in~\cite{Ko23}, see also~\cite{Ko21,LMMV}. First, denote by $r(x)$ the distance function $\dist(p,x)$ in $M$ restricted to a given submanifold $\Sigma^m$. The standard comparison argument, see the proof of~\cite[Lemma~2.1]{Ko21} yields the inequality
$$
-\Delta r(x)\geqslant\coth(r(x))(m-\abs{\nabla r})^2-\abs{H_x},
$$
where $H_x$ is the mean curvature vector at a point $x\in\Sigma^m$, and the sign convention for the Laplacian $\Delta$ on $\Sigma^m$ is chosen so that it is a non-negative operator. As a consequence of this inequality, we see that,  if $r(x)>r$ for a fixed $r>0$, then
\begin{equation}
\label{lap:r}
-\Delta r(x)\geqslant\coth(r(x))(m-\abs{\nabla r})^2-\epsilon_r,
\end{equation}
where $\epsilon_r=\epsilon_r(p)$. Now consider the function
$$
f(t)=\int\limits_0^t\left(\left(\int_0^R\sinh^{m-1}(s)ds\right)/\sinh^{m-1}(R)\right)dR,
$$
defined for $t>0$. A straightforward calculation, using for example~\cite[Lemma~2.5]{Ko21}, shows that it is convex, and satisfies the relation
\begin{equation}
\label{g:in}
f''(t)+(m-1)\coth(t)f'(t)=1.
\end{equation}
Note that the value $1/f'(R)$ is the isoperimetric ratio of the ball $\mathbf B_R^m$ in the hyperbolic space $\mathbf H^m$, and by the classical isoperimetric inequality, coincides with the value of $h(\mathbf{B}^m_R)$. Thus, for a proof of Theorem~\ref{tc} it is sufficient to show that for any admissable open submanifold $\Omega$ in $\Sigma^m\cap (B_R(p)\backslash B_r(p))$ the relation
\begin{equation}
\label{punch}
\frac{\mathit{Area}(\partial\Omega)}{\mathit{Vol}(\Omega)}\geqslant \frac{1}{f'(R)}-\epsilon_r(p)
\end{equation}
holds for all $R>r$. To see this, denote by $\psi$ the function $f\circ r$ defined on the complement $\Sigma^m\backslash B_r(p)$, where $r(x)=\dist(p,x)$. Computing the Laplacian of $\psi$, we obtain
\begin{multline*}
-\Delta\psi=f''(t)\abs{\nabla r}^2-f'(t)\Delta r\geqslant f''(t)\abs{\nabla r}^2+f'(t)\coth(t)(n-\abs{\nabla r}^2)-\epsilon_rf'(t)\\=1+(1-\abs{\nabla r}^2)\left(f'(t)\coth(r)-f''(t)\right)-\epsilon_rf'(t),
\end{multline*}
where $t=r(x)$ and we used relations~\eqref{lap:r}--\eqref{g:in}. Further, using~\cite[Lemma~2.5]{Ko21}, it is straightforward to check that the quantity
$$
f'(t)\coth(t)-f''(t)=m\coth(t)f'(t)-1
$$
is non-negative, and we conclude that 
$$
-\Delta\psi\geqslant 1-\epsilon_rf'(t).
$$
Integrating the last inequality over an admissible open submanifold $\Omega\subset\Sigma^m$ that lies in the annulus $B_R(p)\backslash B_r(p)$, we obtain
\begin{multline*}
(1-\epsilon_rf'(R))\mathit{Vol}(\Omega)\leqslant \int_\Omega (1-\epsilon_rf')d\mathit{Vol}\leqslant -\int_\Omega\Delta\psi d\mathit{Vol}\\ =\int_{\partial\Omega}\langle\grad\psi,\nu\rangle dA\leqslant\int_{\partial\Omega}f'\abs{\nabla r}dA\leqslant f'(R)\mathit{Area}(\partial\Omega),
\end{multline*}
where $\nu$ is a unit normal vector, and we used that $f'>0$ is increasing in the first and the last inequalities above. After elementary transformations, we arrive at inequality~\eqref{punch}.
\qed

\section{Examples of asymptotically minimal submanifolds}
\label{exas}
\subsection{Relationships with $C^2$-regularity}
The purpose of this section is to prove the following statement, which provides a plethora of examples of asymptotically minimal submanifolds. The argument uses the idea reminiscent to one in the proof of~\cite[Lemma~5]{LMMV}.

\begin{prop}
\label{pam}
Let $\Sigma^m$ be a submanifold in the hyperbolic space $\mathbf{H}^{n+1}$ that extends to a $C^2$-smooth submanifold with boundary $\Gamma^{m-1}\subset S^n_\infty(\mathbf{H}^{n+1})$. Then $\Sigma^{m}$ is an asymptotically minimal submanifold of $\mathbf{H}^{n+1}$.
\end{prop}
\begin{proof}
Below we view the hyperbolic space $\mathbf H^{n+1}$ as the unit ball $\mathbb B^{n+1}$ equipped with the hyperbolic metric $g$ that is related to the Euclidean metric $\bar g$ by the formula
$$
\bar g=\phi^2g,\qquad\text{where}\quad \phi(x)=\frac{1}{2}(1-\abs{x}^2),
$$
and $x\in\mathbb B^{n+1}$. Our first aim is to relate the second fundamental forms of $(\Sigma^m,\bar g)$ and $(\Sigma^m,g)$ in appropriate frames, where above we denote by the same symbols the metrics induced by $\bar g$ and $g$ on $\Sigma^m$ respectively. For a given point $x\in\Sigma^m$ we denote by $\bar e_i$, $\bar e_\alpha$, where $i=1,\ldots,m$ and $\alpha=m+1,\ldots,n+1$, a $\bar g$-orthonormal frame such that the $\bar e_i$'s span the tangent space $T_x\Sigma^m$, and the $\bar e_\alpha$'s -- the normal space $N_x\Sigma^m$. Further, straightening $\Sigma^m$ locally near $x$, we see that it extends to a local frame around $x$ such that the $\bar e_i$'s  are tangent to $\Sigma^m$ and all vector fields $\bar e_i$, $\bar e_\alpha$ commute. By $e_i$, $e_\alpha$ we denote the $g$-orthonormal frame, defined by the relations
$$
e_i=\phi\bar e_i,\qquad e_\alpha=\phi\bar e_\alpha.
$$
Now, using Koszul's formula we compute the components of the second fundamental forms:
\begin{multline*}
(\mathit{II}_{\bar g})^{\bar\alpha}_{\bar i \bar j}=\langle\bar\nabla_{\bar e_i}\bar e_j,\bar e_\alpha\rangle_{\bar g}=\frac{1}{2}\left(\bar e_i\langle\bar e_j,\bar e_\alpha\rangle_{\bar g}+\bar e_j\langle\bar e_i,\bar e_\alpha\rangle_{\bar g}-\bar e_\alpha\langle\bar e_i,\bar e_j\rangle_{\bar g}\right)\\=
\phi^{-1}\langle\nabla_{e_i}e_j,e_\alpha\rangle_{g} -\frac{1}{2}\phi^{-1}\left(\langle e_\alpha,[e_i,e_j]\rangle_{g}+\langle e_j,[e_\alpha,e_i]\rangle_g-\langle e_i,[e_j,e_\alpha]\rangle_g\right)\\=
\phi^{-1}(\mathit{II}_g)^\alpha_{ij}-\phi^{-1}e_\alpha(\phi)\delta_{ij}=\phi^{-1}(\mathit{II}_g)^\alpha_{ij}-\bar e_\alpha(\phi)\delta_{ij},
\end{multline*}
where $\delta_{ij}$ is the Kronecker symbol, and by $[\cdot,\cdot]$ we denote the Lie bracket of vector fields. Taking the trace, we obtain the following relation for the mean curvature vectors:
\begin{equation}
\label{h-h}
H^{\bar\alpha}_{\bar g}=\phi^{-1}H^\alpha_g-\bar e_\alpha(\phi)m,
\end{equation}
where $\alpha$ and $\bar\alpha$ take the same values simultaneously, but the use of the bar in superscripts underlines that the coordinates of the vectors are taken in bases $e_\alpha$ and $\bar e_\alpha$ respectively. Thus, after elementary manipulations, we arrive at the inequality
$$
\lvert H_g\rvert^2\leqslant 2\phi^2\left(\lvert H_{\bar g}\rvert^2+m^2\lvert(\overline{\grad}\phi)^\perp\rvert^2\right),
$$
where $\lvert H_g\rvert$ and $\lvert H_{\bar g}\rvert$ are lengths of mean curvature vectors in metrics $g$ and $\bar g$ respectively, and $(\overline{\grad}\phi)^\perp$ denotes the normal component of the gradient of $\phi$ with respect to the Euclidean metric $\bar g$. Note that 
\begin{equation}
\label{grad}
\overline{\grad}\phi(x)=-x,\qquad\text{where}\quad x\in\mathbb B^{n+1},
\end{equation}
and hence, the second term on the right hand-side in the inequality above is bounded. Since $\Sigma^m$ is $C^2$-smooth up to boundary, so is the first term. Finally, since $\phi(x)\to 0$ as $\abs{x}\to 1-$, we conclude that $\Sigma^m$ is asymptotically minimal with respect to the hyperbolic metric $g$ on $\mathbb B^{n+1}$.
\end{proof}

\subsection{Discussion and other consequences}
We end the section with discussing relation~\eqref{h-h}, obtained in the proof of Proposition~\ref{pam}, and its consequences. First, note that the term $\vert H_{\bar g}\rvert^2$ depends only on the first and second derivatives of the immersion $\Sigma^m$ into the Euclidean ball $\mathbb B^{n+1}$, and its behaviour near the boundary $\partial\mathbb B^{n+1}$ can be viewed as some regularity condition on its own. Second, the behaviour of the quantity $\lvert(\overline{\grad}\phi)^\perp\rvert^2$ is also closely related to the property of meeting the boundary orthogonally. More precisely, assuming that a submanifold $\Sigma^m$ in $\mathbb B^{n+1}$ is $C^1$-smooth up boundary in $\partial\mathbb B^{n+1}$, and using relation~\eqref{grad}, it is straightforward to show that the hypothesis that $\lvert(\overline{\grad}\phi)^\perp\rvert^2$ converges to zero as $\abs{x}\to 1-$ is equivalent to the hypothesis that $\Sigma^m$ meets the boundary orthogonally. In fact, the former can be viewed as a more general form for the latter.

As a consequence of this discussion, we state the following corollary, which is somewhat complementary to the content of Theorem~\ref{t1}.
\begin{corollary}
\label{cora}
Let $\Sigma^m$ be a minimal submanifold in the hyperbolic space $\mathbf{H}^{n+1}$. Then it is asymptotically minimal in the Euclidean sense, as a submanifold in $\mathbb B^{n+1}$, if and only if it meets the boundary orthogonally, that is
$$
\sup\{\lvert(\overline{\grad}\phi)^\perp(x)\rvert: x\in \Sigma^m\backslash B^m_\rho(0)\}\longrightarrow 0\qquad\text{as}\quad\rho\to 1-.
$$
In particular, if $\Sigma^m$ is $C^1$-smooth up to boundary at infinity, then it is asymptotically minimal in the Euclidean sense.
\end{corollary}

Conversely, given a submanifold $\Sigma^m\subset\mathbb B^{n+1}$ with certain Euclidean properties, we are able to relate the rate of decay of the hyperbolic mean curvature to the property of meeting the boundary orthogonally.
\begin{corollary}
Let $\Sigma^m\subset\mathbb B^{n+1}$ be an asymptotically minimal  submanifold in the Euclidean sense. Then it is asymptotically minimal as a submanifold in the hyperbolic space $\mathbf{H}^{n+1}$ in the unit ball model. Moreover, if $\lvert H_g\rvert/\phi\to 0+$ as $\abs{x}\to 1-$, then $\Sigma^m$ meets the boundary at infinity orthogonally.
\end{corollary}
The proofs are direct consequences of relations~\eqref{h-h}--\eqref{grad}. Theorem~\ref{t1} is also used for the last statement in Corollary~\ref{cora}.

\section{Proof of Theorem~\ref{t2}}
\label{proofs}

\subsection{Outline of the strategy}
The proof of Theorem~\ref{t2} follows closely the line of argument in the proof of~\cite[Theorem~1.3]{Ko23}. More precisely, by Corollary~\ref{ess} and the standard spectrum classification for self-adjoint operators, see~\cite{BD}, it is sufficient to show that any real number $\lambda>(m-1)^2/4$ belongs to the spectrum of $\Sigma^m$. For the latter it is, in turn, sufficient to construct a sequence $(\phi_k)$ of smooth functions with disjoint compact supports,  and a sequence of positive real numbers $(\epsilon_k)$, $\epsilon_k\to 0+$, such that
\begin{equation}
\label{r:aim}
\int_{\Sigma^m}\abs{\Delta\phi_k-\lambda\phi_k}^2d\mathit{Vol}\leqslant\epsilon_k\int_{\Sigma^m}\abs{\phi_k}^2d\mathit{Vol}\qquad\text{for all~} k\in\mathbb N.
\end{equation}
The strategy is first to construct appropriate radial test-functions on the ambient space $\mathbf{H}^{n+1}$, and use them to compute the spectrum of cones. Then, we perform a version of this argument on $\Sigma^m$, controlling the deviation of the estimates from the ones on the cone. Below we describe the necessary ingredients for this in more detail.

\subsection{Test-functions and spectrum of cones}
For the rest of the section we assume that $\Sigma^m$ is an asymptotically minimal submanifold that is regular at infinity -- for a submanifold $\Gamma^{m-1}\subset S^n_\infty(\mathbf{H}^{n+1})$ the union $\Sigma^m\cup\Gamma^{m-1}$ is a $C^1$-smooth submanifold with boundary that meets the boundary at infinity orthogonally. Let $C^m$ be a cone defined by the following relation
\begin{equation}
\label{cone:def}
C^m=\{\tau z: z\in\Gamma^{m-1}, \tau\in (0,1]\}\subset\mathbb B^{n+1}.
\end{equation}
By the radial function we mean a function that depends on the distance $r(x)=\dist(x,p)$ only, where $p\in\mathbf{H}^{n+1}$ is a point that corresponds to the origin in the unit ball $\mathbb B^{n+1}$. We start with the following well-known fact, see~\cite{Ko23}, which for the convenience of references, we state as a lemma. 
\begin{lemma}
\label{laplace:cone}
Let $f$ be a smooth radial function on $\mathbf{H}^{n+1}$. Then the Laplacian of its restriction to the cone $C^m$ is given by the formula
$$
-\Delta f(x)=f''(r(x))+(m-1)\coth(r(x))f'(r(x)),
$$
where $r(x)=\dist(x,p)$ and $x\in C^m$.
\end{lemma}
Now for a given $\lambda>(m-1)^2/4$ we consider the complex-valued function
$$
\psi(t)=\sinh^{-(m-1)/2}(t)\exp(i\beta t),\qquad\text{where~}\beta=\sqrt{\lambda-\frac{1}{4}(m-1)^2},
$$
and $t$ ranges in the interval $(0,+\infty)$. A straightforward calculation shows that it satisfies the relation
\begin{equation}
\label{psi}
\psi''(t)+(m-1)\coth(t)\psi'(t)+(\lambda+\alpha(t))\psi(t)=0,
\end{equation}
where by $\alpha$ we denote the function
$$
\alpha(t)=\frac{1}{4}(m-1)(m-3)\sinh^{-2}(t).
$$
Further, for a given $R>0$ we define the model function $\upsilon_R$ on $(0,+\infty)$ by the formula
$$
\upsilon_R(t)=\left\{
\begin{array}{ll}
\psi(t)\sin^2\left(\frac{2\pi}{R}\left(t-\frac{R}{2}\right)\right), & \text{if~ } t\in[R/2,R];\\
0, & \text{otherwise}.
\end{array}
\right.
$$
The following lemma, whose proof can be found in~\cite{Ko23}, establishes integral estimates for the derivatives of $\upsilon_R$, which play the key role in the sequel.
\begin{lemma}
\label{est}
For any $R>0$ the function $\upsilon_R$, defined above, satisfies the inequality
$$
\int\limits_{R/2}^{R}\abs{\upsilon_{R}''(t)+(m-1)\coth(t)\upsilon_{R}'(t)+\lambda\upsilon_{R}(t)}^2\sinh^{m-1}(t)dt\leqslant \epsilon_R\int\limits_{R/2}^{R}\abs{\upsilon_R(t)}^2\sinh^{m-1}(t)dt,
$$
where 
$$
\epsilon_R=2\max\left\{\alpha^2\left(\frac{R}{2}\right), 16\abs{\beta}^2\left(\frac{2\pi}{R}\right)^2, 16\left(\frac{2\pi}{R}\right)^4\right\}.
$$
In addition, for a fixed $R_0>0$ and any $R>R_0$ the following inequalities hold:
\begin{equation}
\label{est:in1}
\int\limits_{R/2}^{R}\abs{\upsilon'_R(t)}^2\sinh^{m-1}(t)dt\leqslant C_*\int\limits_{R/2}^{R}\abs{\upsilon_R(t)}^2\sinh^{m-1}(t)dt,
\end{equation}
\begin{equation}
\label{est:in2}
\int\limits_{R/2}^{R}\abs{\upsilon''_R(t)}^2\sinh^{m-1}(t)dt\leqslant C_{*}\int\limits_{R/2}^{R}\abs{\upsilon_R(t)}^2\sinh^{m-1}(t)dt,
\end{equation}
where the constant $C_*$ depends on $m$, $\lambda$, and $R_0$.
\end{lemma}

Now we explain how, constructing test-functions, one shows that any $\lambda>(m-1)^2/4$ lies in the spectrum of the cone $C^m$. Following the argument in~\cite{Ko23}, choose a sequence $R_k\to +\infty$ such that $R_{k+1}>2R_k$, and denote by $\upsilon_k$ compactly supported smooth functions on the real line whose supports are mutually disjoint, and such that the following inequalities hold:
\begin{multline}
\label{aprox:1}
\int\limits_{0}^{+\infty}\abs{\upsilon_{k}''(t)+(m-1)\coth(t)\upsilon_{k}'(t)+\lambda\upsilon_{k}(t)}^2\sinh^{m-1}(t)dt \\ \leqslant 2\int\limits_{R_k/2}^{R_k}\abs{\upsilon_{R_k}''(t)+(m-1)\coth(t)\upsilon_{R_k}'(t)+\lambda\upsilon_{R_k}(t)}^2\sinh^{m-1}(t)dt,
\end{multline}

\begin{equation}
\label{aprox:2}
\int\limits_{R_k/2}^{R_k}\abs{\upsilon_{R_k}}^2\sinh^{m-1}(t)dt\leqslant 2\int\limits_{0}^{+\infty}\abs{\upsilon_{k}}^2\sinh^{m-1}(t)dt.
\end{equation}
Such smooth functions $\upsilon_k$ can be constructed as approximations of $W^{2,2}$-functions $\upsilon_{R_k}$, for example, by using the mollification technique. Denote by $\phi_k$ the restriction of the function $\upsilon_k\circ r$ to the cone $C^m$. Then, by Lemma~\ref{laplace:cone} we obtain
$$
\int_{C^m} \abs{\Delta\phi_k-\lambda\phi_k}^2d\mathit{Vol}=\omega(\Gamma^{m-1})\int\limits_{0}^{+\infty}\abs{\upsilon_{k}''(t)+(m-1)\coth(t)\upsilon_{k}'(t)+\lambda\upsilon_{k}(t)}^2\sinh^{m-1}(t)dt, 
$$
where $\omega(\Gamma^{m-1})$ is the volume of $\Gamma^{m-1}$ in the standard metric on the unit sphere $\partial\mathbb B^{n+1}$. Finally, combining the above with Lemma~\ref{est} and relations~\eqref{aprox:1}--\eqref{aprox:2}, we arrive at the inequality
\begin{equation}
\label{f:cone}
\int_{C^m} \abs{\Delta\phi_k-\lambda\phi_k}^2d\mathit{Vol}\leqslant\epsilon_k\int_{C^m}\abs{\phi_k}^2d\mathit{Vol},
\end{equation}
where $\epsilon_k=4\epsilon_{R_k}$, and $\epsilon_{R_k}$ is the quantity from Lemma~\ref{est}.  Since $\epsilon_{R_k}\to 0+$, we conclude that such a real number $\lambda$ indeed lies in the spectrum of the cone $C^m$.

\subsection{Controlling the deviation from the cone}
To perform an argument similar to the one above on an asymptotically minimal submanifold $\Sigma^m$, we need two auxiliary lemmas. The first is a version of the statement in~\cite{Ko23}, where it is proved for singular minimal submanifolds.
\begin{lemma}
\label{laplace:current}
Let $\Sigma^m$ be an asymptotically minimal submanifold that is regular at infinity. Then for any $\varepsilon>0$ there exists $R>0$ such that for any smooth radial function $f$ that is supported in the complement of the hyperbolic ball $\mathbf{B}_{R}(p)$ the Laplacian of its restriction to $\Sigma^m$ is given by the formula
$$
-\Delta f(x)=f''(r(x))+(m-1)\coth(r(x))f'(r(x)) +E(x),
$$
where the error term $E(x)$ satisfies the estimate
$$
\abs{E(x)}\leqslant\varepsilon\left(\abs{f''(r(x))}+\abs{f'(r(x))}\right),
$$
$r(x)=\dist(x,p)$, and $x\in\Sigma^m$.
\end{lemma}
The proof of a similar statement in~\cite{Ko23} carries over asymptotically minimal submanifolds that are regular at infinity. However, for reader's convenience, we outline the key points below.
\begin{proof}[Sketch of the proof]
Since $\Sigma^m$ meets the boundary at infinity orthogonally, for a sufficiently large $r=r(x)$ the tangent spaces $T_x\Sigma^m$ and $T_x\partial\mathbf{B}_r(p)$ span $T_x\mathbf{H}^{n+1}$, and we conclude that $\Sigma^m$ and $\partial\mathbf{B}_r(p)$ intersect transversally. Thus, the intersection $\Sigma^m\cap\partial\mathbf{B}_r(p)$ is a submanifold in $\partial\mathbf{B}_r(p)$, and it is straightforward to construct vector fields $X_i$ in a neighbourhood of $x$ in $\mathbf{H}^{n+1}$, where $i=1,\ldots, n$, such that at the point $x$ they form an orthogonal system in $T_x\mathbf{B}_r(p)$, and the first $(m-1)$ vectors $X_1,\ldots, X_{m-1}$ at $x$ lie in the tangent space $T_x\Sigma^m$. Now we define the vector $\tilde X_r\in T_x\Sigma^m$ as a unit vector that is orthogonal to these vectors and such that the product $\langle \tilde X_r,X_r\rangle$ is positive, where $X_r$ is the radial vector field $\partial/\partial r$ in $\mathbf{H}^{n+1}$. This vector $\tilde X_r$ can be extended to a vector field around $x$ such that
\begin{equation}
\label{def:xr}
\abs{\tilde X_r}^2=1\qquad\text{and}\qquad\langle\tilde X_r,X_i\rangle=0\quad\text{for all}\quad i=1,\ldots,m-1.
\end{equation}
As is explained in the proof of~\cite[Lemma~4.4]{Ko23}, the hyperbolic length $\abs{\tilde X_r-X_r}$ tends to zero as $r=r(x)$ tends to infinity. Thus, writing $\tilde X_r$ in the form
\begin{equation}
\label{decom}
\tilde X_r=\varphi X_r+Y_r,\qquad\text{where}\quad\langle X_r,Y_r\rangle=0,
\end{equation}
we conclude that the function $(1-\varphi^2)$ and the hyperbolic length $\abs{Y_r}$ tend to zero when $r=r(x)$ tends to infinity. A standard computation, which can be found in~\cite{Ko21}, shows that the Laplacian of the restriction of $f$ to $\Sigma^m$ at $x$ is now given by the formula
$$
-\Delta f(x)=\hess_x f(\tilde X_r,\tilde X_r)+\frac{1}{\sinh^2(r(x))}\sum_{i=1}^{m-1}\hess_x f(X_i,X_i)+f'(r)\langle X_r,H_x\rangle,
$$
where $H_x$ is the mean curvature at $x\in\Sigma^m$. Again, following the lines in the proof of~\cite[Lemma~4.4]{Ko23}, we see that the error term $E(x)$ in the statement of the lemma, has the form
$$
E(x)=(1-\varphi^2)f''(r)-\langle X_r,\nabla_{Y_r}Y_r\rangle f'(r)+ f'(r)\langle X_r,H_x\rangle.
$$
From the above we already know that the function $(1-\varphi^2)$ is small for a large $r=r(x)$. Besides, since $\Sigma^m$ is asymptotically minimal, so is the term $\langle X_r,H_x\rangle$ on the right hand-side above. As is shown in the proof of~\cite[Lemma~4.4]{Ko23}, the product $\langle X_r,\nabla_{Y_r}Y_r\rangle$ also tends to zero as $r=r(x)\to+\infty$, and thus, we are done.
\end{proof}

Finally, we also need the following lemma, see~\cite[Lemma~4.1]{Ko23}.
\begin{lemma}
\label{vol}
Let $\Sigma^m$ be a submanifold in the hyperbolic space $\mathbf{H}^{n+1}$ that is asymptotic to $\Gamma^{m-1}\subset S^n_\infty(\mathbf{H}^{n+1})$ and regular at infinity, and let $C^m$ be the cone given by~\eqref{cone:def}. Then, for any $0<\varepsilon<1$ there exists $R>0$ such that for any non-negative bounded radial function $f$ with compact support in the complement of the hyperbolic ball $\mathbf{B}_{R}(p)$ the inequalities 
$$
(1-\varepsilon)\!\!\!\int_{C^m} fd\mathit{Vol}\leqslant\int_{\Sigma^m} fd\mathit{Vol}\leqslant (1+\varepsilon)\!\!\!\int_{C^m} fd\mathit{Vol}
$$
hold.
\end{lemma}

Now we are able to finish the proof of Theorem~\ref{t2}. Following the line of argument in~\cite{Ko23}, for a given $0<\varepsilon<1$ let $R>0$ be a real number that satisfies the conclusions of  Lemmas~\ref{est} -- \ref{vol}. For a sequence $R_k\to +\infty$ such that $R_{k+1}>2R_k>2R$, we can choose compactly supported smooth functions  $\upsilon_k$ defined on the real line whose supports  are mutually disjoint, such that they satisfy relations~\eqref{aprox:1}-\eqref{aprox:2}. Since such functions are constructed as $W^{2,2}$-approximations of the $\upsilon_{R_k}$'s, we may also assume that they satisfy the inequalities 
\begin{equation}
\label{aprox:3}
\int\limits_{0}^{+\infty}\abs{\upsilon'_{k}}^2\sinh^{m-1}(t)dt\leqslant 4C_*\int\limits_{0}^{+\infty}\abs{\upsilon_{k}}^2\sinh^{m-1}(t)dt,
\end{equation}
\begin{equation}
\label{aprox:4}
\int\limits_{0}^{+\infty}\abs{\upsilon''_{k}}^2\sinh^{m-1}(t)dt\leqslant 4C_*\int\limits_{0}^{+\infty}\abs{\upsilon_{k}}^2\sinh^{m-1}(t)dt,
\end{equation}
where $C_*$ is the constant from Lemma~\ref{est}. Now denote by $\tilde\phi_k$ the restriction of the function $\upsilon_k\circ r$ to the submanifold $\Sigma^m$. Then, by Lemma~\ref{laplace:current} we obtain
\begin{multline}
\label{f:source}
\int_{\Sigma^m} \abs{\Delta\tilde\phi_k-\lambda\tilde\phi_k}^2d\mathit{Vol}\leqslant
3\int_{\Sigma^m} \abs{\upsilon''_k+(m-1)\coth(\cdot)\upsilon'_k+\lambda\upsilon_k}^2 d\mathit{Vol}\\
+3\varepsilon^2\int_{\Sigma^m} \abs{\upsilon''_k}^2+\abs{\upsilon'_k}^2 d\mathit{Vol}.
\end{multline}
To estimate the first integral on the right hand-side above we use Lemmas~\ref{laplace:cone} and~\ref{vol}, and already obtained estimate~\eqref{f:cone} for the cone $C^m$. Combining all these ingredients, we get
\begin{multline*}
\int_{\Sigma^m} \abs{\upsilon''_k+(m-1)\coth(\cdot)\upsilon'_k+\lambda\upsilon_k}^2 d\mathit{Vol}\\
\leqslant(1+\varepsilon)\int_{C^m} \abs{\upsilon''_k+(m-1)\coth(\cdot)\upsilon'_k+\lambda\upsilon_k}^2 d\mathit{Vol}=(1+\varepsilon)\int_{C^m} \abs{\Delta\phi_k-\lambda\phi_k}^2d\mathit{Vol}\\
\leqslant\varepsilon_k(1+\varepsilon)\int_{C^m}\abs{\phi_k}^2d\mathit{Vol}
\leqslant\varepsilon_k\frac{1+\varepsilon}{1-\varepsilon}\int_{\Sigma^m}\abs{\tilde\phi_k}^2d\mathit{Vol}.
\end{multline*}
Here, we follow the notation used above, denoting by $\phi_k$ the restriction of the radial function $\upsilon_k$ to the cone $C^m$, and $\varepsilon_k=4\varepsilon_{R_k}$, where $\varepsilon_R$ is given by the relation in Lemma~\ref{est}. To estimate the second integral on the right hand-side of relation~\eqref{f:source}, we use integral inequalities~\eqref{aprox:3} and~\eqref{aprox:4} for the first two derivatives of $\upsilon_k$. In more detail, we can bound the $L^2$-norm of $\upsilon'_k\circ r$ on $\Sigma^m$ in the following fashion:
\begin{multline*}
\int_{\Sigma^m}\abs{\upsilon'_k}^2 d\mathit{Vol}\leqslant (1+\varepsilon)\int_{C^m}\abs{\upsilon'_k}^2 d\mathit{Vol}
\\=(1+\varepsilon)\omega(\Gamma^{m-1})\int\limits_0^{+\infty}\abs{\upsilon'_k(t)}^2\sinh^{m-1}(t)dt\\
\leqslant 4C_*(1+\varepsilon)\omega(\Gamma^{m-1})\int\limits_0^{+\infty}\abs{\upsilon_{k}(t)}^2\sinh^{m-1}(t)dt\\ =4C_*(1+\varepsilon)\int_{C^m}\abs{\upsilon_k}^2 d\mathit{Vol}
\leqslant 4C_*\frac{1+\varepsilon}{1-\varepsilon}\int_{\Sigma^m}\abs{\tilde\phi_k}^2 d\mathit{Vol},
\end{multline*}
where $\omega(\Gamma^{m-1})$ is the volume of $\Gamma^{m-1}$ in the standard metric in the unit sphere $\partial\mathbb B^{n+1}$. Similarly, we get the following estimate for the $L^2$-norm of $\upsilon_k''\circ r$:
$$
\int_{\Sigma^m}\abs{\upsilon''_k}^2 d\mathit{Vol}\leqslant 4C_*\frac{1+\varepsilon}{1-\varepsilon}\int_{\Sigma^m}\abs{\tilde\phi_k}^2 d\mathit{Vol}.
$$
Combining all these inequalities with relation~\eqref{f:source}, we finally obtain
$$
\int_{\Sigma^m} \abs{\Delta\tilde\phi_k-\lambda\tilde\phi_k}^2d\mathit{Vol}\leqslant
(3\varepsilon_k+24C_*\varepsilon^2)\frac{1+\varepsilon}{1-\varepsilon}\int_{\Sigma^m}\abs{\tilde\phi_k}^2 d\mathit{Vol}.
$$
Since $0<\epsilon<1$ is arbitrary and $R_k\to+\infty$, choosing a subsequence $\tilde\phi_{k_\ell}$, we may assume that
$$
\int_{\Sigma^m} \abs{\Delta\tilde\phi_{k_\ell}-\lambda\tilde\phi_{k_\ell}}^2d\mathit{Vol}\leqslant
(3\varepsilon_{k_\ell}+24C_*/\ell^{2})\frac{1+1/\ell}{1-1/{\ell}}\int_{\Sigma^m}\abs{\tilde\phi_{k_\ell}}^2 d\mathit{Vol}
$$
for all integers $\ell>1$. Since $\varepsilon_k=4\varepsilon_{R_k}\to 0+$, by Lemma~\ref{est},  we are done.
\qed

\end{document}